\documentclass[12]{amsart}
\usepackage{amsmath,amssymb,amsthm,color,enumerate,comment,centernot,enumitem,url,cite}
\usepackage{graphicx,relsize,bm}
\usepackage{mathtools}
\usepackage{array}

\makeatletter
\newcommand{\tpmod}[1]{{\@displayfalse\pmod{#1}}}
\makeatother

\newtheorem{thm}{Theorem}[section]

\newtheorem{cor}[thm]{Corollary}

\theoremstyle{remark}

\theoremstyle{definition}

\theoremstyle{THM}

\newcommand{\CC}{{\mathcal C}}
\newcommand{\W}{{\mathcal W}}

\newcommand{\SSS}{{\mathcal S}}
\newcommand{\RR}{{\mathcal R}}
\newcommand{\FF}{{\mathcal F}}
\newcommand{\Gal}{{\mbox{{\rm{Gal}}}}}

\newcommand{\mmod}[1]{\ \mathrm{mod}\enspace #1}
\newcommand{\Z}{{\mathbb Z}}
\newcommand{\Q}{{\mathbb Q}}

\newcommand{\F}{{\mathbb F}}

\makeatletter
\@namedef{subjclassname@2020}{%
  \textup{2020} Mathematics Subject Classification}
\makeatother

\def\red#1 {\textcolor{red}{#1 }}
\def\blue#1 {\textcolor{blue}{#1 }}

\numberwithin{equation}{section}

\begin{document}

\title[Monogenic trinomials ${x^{12}+ax^6+b}$]{Characterizing monogenic trinomials $\boldsymbol{x^{12}+ax^6+b}$ according to their Galois groups}


\author{Lenny Jones}
\address{Professor Emeritus, Department of Mathematics, Shippensburg University, Shippensburg, Pennsylvania 17257, USA}
\email[Lenny~Jones]{doctorlennyjones@gmail.com}

\date{\today}

\begin{abstract}
    Let $f(x)=x^{12}+ax^{6}+b\in {\mathbb Z}[x]$, with $ab\ne 0$. We say that $f(x)$ is {\em monogenic} if $f(x)$ is irreducible over ${\mathbb Q}$ and
    $\{1,\theta,\theta^2,\ldots,\theta^{11}\}$
    is a basis for the ring of integers of ${\mathbb Q}(\theta)$, where $f(\theta)=0$. For each possible Galois group $G$ of $f(x)$ over ${\mathbb Q}$, we give explicit descriptions of all monogenic trinomials $f(x)$ having Galois group $G$. These results extend recent work on monogenic power-compositional quartic and sextic trinomials.
   \end{abstract}


\maketitle
\section{Introduction}\label{Section:Intro}

We say that a monic polynomial $F(x)\in \Z[x]$ is {\em monogenic} if $F(x)$ is irreducible over $\Q$ and  $\{1,\theta,\theta^2,\ldots,\theta^{\deg(F)-1}\}$
  is a basis for the ring of integers, $\Z_K$, of $K={\mathbb Q}(\theta)$, where $F(\theta)=0$. 
  It is well known \cite{Cohen} that 
  \begin{equation} \label{Eq:Dis-Dis}
\Delta(F)=\left[\Z_K:\Z[\theta]\right]^2\Delta(K),
\end{equation} where $\Delta(F)$ and $\Delta(K)$ denote the discriminants over $\Q$, respectively, of $F(x)$ and the number field $K$. We refer to $\left[\Z_K:\Z[\theta]\right]$ as the {\em index} of $F(x)$, and we denote it index($F$). 
Thus, from \eqref{Eq:Dis-Dis}, 
\begin{equation}\label{Mono}
F(x) \ \mbox{is monogenic if and only if} \ \Delta(F)=\Delta(K),\  \mbox{or equivalently,} \  \Z_K=\Z[\theta].
\end{equation}

 Throughout this article, for $a,b\in \Z$ with $ab\ne 0$ and $j\in \{1,2,3\}$, we let 
 \begin{align}\label{Basics}
\begin{split}
 f(x):&=x^{12}+ax^6+b, \quad g_{2j}(x):=x^{2j}+ax^j+b, \quad r(x):=x^3-3bx+ab,\\
 \alpha:&=a+2\sqrt{b}, \quad \beta:=a-2\sqrt{b}, \quad \delta:=a^2-4b, \quad \W:=\delta/\gcd(2,a)^2.
\end{split}
\end{align}
Observe that if $f(x)$ is irreducible over $\Q$, then the trinomials $g_{2j}(x)$ are irreducible over $\Q$ as well. Note also that $r(x)$ is a cubic resolvent for $g_6(x)$.

 We use the notation 4T$n$, 6T$n$ and 12T$n$ \cite{BM}, to denote the possible Galois groups $G_4$, $G_6$ and $\Gal(f)$, 
for respectively,  the trinomials $g_4(x)$, $g_6(x)$ and $f(x)$. These possible groups \cite{KW,HJActa,Chen} are given in Table \ref{T:PG}. 
\begin{table}[h]  
 \begin{center}
\begin{tabular}{c|cccccccccccccccc}
 & \multicolumn{16}{c}{$n$} \\ \hline
  4T$n$ & 1 & 2 & 3 & \\ [.25em]
 6T$n$ & 1 & 2 & 3  & 5 & 9 \\  [.25em]
 12T$n$ &  2 & 3 & 10 & 11 & 12 & 13 & 14 & 15 & 16 & 18 & 28 & 37 & 38 & 39 & 42 & 81
 \end{tabular}
\end{center}
\caption{Possible Galois groups for $g_4(x)$, $g_6(x)$ and $f(x)$}
 \label{T:PG}
\end{table}\\
For integers $r$ and $N$, with $N>1$, we let the notation ``$r \mmod{N}$" denote the unique integer $z\in \{0,1,2,\ldots,N-1\}$ such that $r\equiv z \pmod{N}$. That is, $r \mmod{N}=z$.

In this article, we provide a complete characterization of the monogenic trinomials $f(x)$ according to $\Gal(f)\simeq$ 12T$n$ for each value of $n$ given in Table \ref{T:PG}.  More precisely, we prove the following: 
\begin{thm}\label{Thm:Main}
Suppose that $f(x)$ is irreducible over $\Q$, and that $\Gal(f)\simeq$ $12{\rm T}n$, for some $n$ in Table \ref{T:PG}. If $n\in \{3,11,12,13,14,15,16,18,37\}$, then $f(x)$ is not monogenic. If $n\in \{2,10,28,38,39,42,81\}$, then $f(x)$ is monogenic with $\Gal(f)\simeq$ 
\begin{enumerate}
  \item \label{M:I2} $12{\rm T}2$ if and only if $f(x)=x^{12}-x^6+1$; 
  \item \label{M:I10} $12{\rm T}10$ if and only if $f(x)=x^{12}+ax^6+1$, where $a$ satisfies the conditions 
 \[\CC_{10}:=\{a\ne -1,\  a \mmod{4}\in \{0,3\},\ a\mmod{9}\ne 0, \  \W \ \mbox{is squarefree}\};\]
 \item \label{M:I28} $12{\rm T}28$ if and only if either $f(x)\in \{x^{12}+2x^6+2,\ x^{12}-2x^6+2\}$ or\\
    $f(x)=x^{12}+ax^6-1$, where $a$ satisfies the conditions
     \[\CC_{28}:=\{a \mmod{4}\ne 0,\ a\mmod {9}\not \in \{0,4,5\},\ \W \ \mbox{is squarefree}\};\]
   \item \label{M:I38} $12{\rm T}38$ if and only if 
   \[f(x)\in \{x^{12}\pm 4x^6-2,\ x^{12}\pm 4x^6+6\}\] or $f(x)=x^{12}+ax^6-3$, where $a$ satisfies the conditions 
   \[\CC_{38}:=\{a \mmod{4}\in \{0,3\},\ a\mmod{9}\not \in \{2,7\},\ \W \ \mbox{is squarefree}\};\]
   \item \label{M:I39} $12{\rm T}39$ if and only if $f(x)\in \{x^{12}\pm 4x^6+2,\ x^{12}-5x^6+5\}$.
   \item \label{M:I42} $12{\rm T}42$ if and only if $f(x)=x^{12}+ax^6+(a^2+3)/4$, where $a$ satisfies the conditions 
    \[\CC_{42}:=\{a\mmod{8}\in \{3,5,7\}, \ a\ne -1,\ (a^2+3)/4 \ \mbox{is squarefree}\};\]
   \item \label{M:I81} $12{\rm T}81$ if and only if $a$ and $b$ satisfy the conditions
   \[\CC_{81}:=\{b\ne 1, \ b \ \mbox{and} \ \W \ \mbox{are squarefree}, \ -3b, \ -3b\delta \ \mbox{and} \ b\delta \ \mbox{are not squares}\},\]
   such that $(a \mmod{4}, \ b \mmod{4})\in \RR$ and $(a \mmod{9}, \ b \mmod{9})\in \SSS$, where
    \[\RR=\{(0,1),(0,2),(2,2),(2,3),(1,3),(3,2),(3,1),(3,3)\} \ \mbox{and}\] $\SSS=\SSS_1 \cup \SSS_2\cup \SSS_3\cup\SSS_4$, with 
    $\SSS_1=\{(0,3),(0,6),(3,3),(3,6),(6,3),(6,6)\}$, 
    \begin{align*}
    \SSS_2:&=\{(0,2),(0,4),(0,5),(0,7),(3,1),(3,4),\\
    & \qquad (3,7),(3,8),(6,1),(6,4),(6,7),(6,8)\},
    \end{align*} 
    \[\SSS_3:=\{(1,3),(1,6),(2,3),(4,6),(5,6),(7,3),(8,3),(8,6)\}\] and 
    \begin{align*}
   \SSS_4:&=\{(1,1),(1,2),(1,4),(1,5),(1,8), (2,2),(2,4),(2,5),(2,7),(2,8),\\
     & \qquad (4,1),(4,2),(4,5),(4,7), (5,1),(5,2),(5,5),(5,7),(7,2),(7,4), \\
    & \qquad (7,5),(7,7),(7,8),(8,1),(8,2),(8,4),(8,5),(8,8)\}.
   \end{align*}
  \end{enumerate} 
\end{thm}
\noindent We point out that similar research involving the monogenicity and/or Galois groups of trinomials of various degrees has been conducted by many authors  \cite{AJ,AL,BS,HJJAA,HJBAMS,HJActa,HJIMPCT,HJOMGG,JonesQuarticsBAMS,JonesRam,JonesNYJM,JonesJAA,JonesAA,
JonesEvenSextics,JonesMCCT,JonesMEQT,JonesMEOPGG,JW,MNSU,S1,S2,Voutier}.

The proof of Theorem \ref{Thm:Main} relies first on the characterization of the Galois groups of $f(x)$ \cite{Chen}. The next step is to examine the monogenicity of $f(x)$, which is typically done by using a theorem due to Jakhar, Khanduja and Sangwan \cite{JKS2} that gives necessary and sufficient conditions to decide when a prime divisor of the discriminant does not divide index($f$). However, we are able to streamline this process here in the following ways. 
First, Chen \cite{Chen} has provided a characterization of the Galois groups of $f(x)$ in terms of the Galois groups of $g_4(x)$ \cite{KW} and $g_6(x)$ \cite{Cav,HJMS}, along with some additional restrictions on the coefficients $a$ and $b$ of $f(x)$. Second, a recent theorem due to Kaur, Kumar and Remete \cite{KKR} (see Theorem \ref{Thm:KKR}) that is designed specifically for determining the monogenicity of power-compositional polynomials, allows us to utilize the characterizations of the monogenic trinomials $g_{2j}(x)$. Some, but not all, of the foundation for this step has already been established. The characterization of the monogenic trinomials $g_6(x)$ has been given recently in \cite{HJActa}. However, only a partial characterization of the monogenic trinomials $g_4(x)$ has been given in \cite{HJBAMS,JonesQuarticsBAMS}, and so we complete this characterization here. For the sake of completeness, we also provide the characterization of the monogenic trinomials $g_2(x)$.   In many cases, this step narrows down considerably the possibilities for monogenic trinomials $f(x)$. Finally, we still need to check the monogenicity of the trinomials $f(x)$ that have survived up to this point, and we use a special case of the theorem of Jakhar, Khanduja and Sangwan \cite{JKS2} tailored for our specific situation (see Theorem \ref{Thm:JKS}) to accomplish this task. Another component of Theorem \ref{Thm:KKR}, which allows us to focus on the primes $q\in \{2,3\}$ when invoking Theorem \ref{Thm:JKS}, also helps to expedite the entire process.

\section{Preliminaries}\label{Section:Prelim}
Let $m\ge 1$ be an integer, and let ${\mathfrak F}(x)=x^{2m}+Ax^m+B\in \Z[x]$. Then, from a theorem due to Swan \cite{Swan}, we have that 
\begin{equation}\label{Dis} 
\Delta({\mathfrak F})=B^{m-1}m^{2m}(A^2-4B)^m. 
\end{equation} 
Note that since $\Delta(g_2)=a^2-4b=\delta$, we have from \eqref{Dis}, that 
\begin{equation*}\label{Dis Here} \Delta(g_4)=2^4\delta^2b,\quad \Delta(g_6)=3^6\delta^3b^2 \quad \mbox{and} \ \Delta(f)=2^{12}3^{12}\delta^6b^5.
\end{equation*} 

The following theorem is a special case of a theorem due to Jakhar, Khanduja and Sangwan \cite{JKS2} that is tailored to our specific situation. 
\begin{thm}{\rm \cite{JKS2}}\label{Thm:JKS}
Let $m\ge 1$ be an integer, and let 
${\mathfrak F}(x)=x^{2m}+Ax^m+B\in \Z[x]$ be irreducible over $\Q$. Let $\Z_K$ denote the ring of integers of $K=\Q(\theta)$, where ${\mathfrak F}(\theta)=0$. 
A prime factor $q$ of $\Delta({\mathfrak F})$ does not divide $\left[\Z_K:\Z[\theta]\right]$ if and only if $q$ satisfies one of the following conditions:
\begin{enumerate}
  \item \label{JKS:C1} when $q\mid A$ and $q\mid B$, then $q^2\nmid B$;
  \item \label{JKS:C2} when $q\mid A$ and $q\nmid B$, then
  \[\mbox{either } \quad q\mid a_2 \mbox{ and } q\nmid b_1 \quad \mbox{ or } \quad q\nmid a_2\left(a_2^2B+b_1^2\right),\]
  where $a_2=A/q$ and $b_1=\frac{B+(-B)^{q^j}}{q}$, such that $q^j\mid\mid 2m$ with $j\ge 1$;
  \item \label{JKS:C3} when $q\nmid A$ and $q\mid B$, then
  \[\mbox{either } \quad q\mid a_1 \mbox{ and } q\nmid b_2 \quad \mbox{ or } \quad q\nmid a_1b_2^{m-1}\left(Aa_1-b_2\right),\]
  where $b_2=B/q$ and $a_1=\frac{A+(-A)^{q^l}}{q}$, such that $q^l\mid\mid m$ with $l\ge 0$;
  \item \label{JKS:C4} when $q\nmid AB$ and $q\mid m$ with $m=sq^k$, $q\nmid s$, then the polynomials 
   \begin{equation*}
     H_1(x):=x^{2s}+Ax^s+B \quad \mbox{and}\quad H_2(x):=\dfrac{Ax^{sq^k}+B+\left(-Ax^s-B\right)^{q^k}}{q} 
   \end{equation*}
   are coprime in $\F_q[x]$; 
         \item \label{JKS:C5} when $q\nmid ABm$, then $q^2\nmid \left(A^2-4B\right)$.
   \end{enumerate}
\end{thm}

We now use Theorem \ref{Thm:JKS} to present a complete description of the monogenic trinomials $g_2(x)$. Of course, in this situation, we have that $\Gal(g_2)\simeq 2{\rm T}1$, the cyclic group of order 2, for all irreducible trinomials $g_2(x)$. 
\begin{thm}\label{Thm:Main g_2} The trinomial $g_2(x)$ is monogenic if and only if $\W$ is squarefree and 
 \[(a \mmod{4}, \ b\mmod{4})\in \{(0,1),(0,2),(2,2),(2,3)\} \quad \mbox{if $2\mid a$.}\]
\end{thm}
\begin{proof}
   Let $q$ be a prime divisor of $\Delta(g_2)=a^2-4b=\delta$.  Observe that conditions \eqref{JKS:C3} and \eqref{JKS:C4} of Theorem \ref{Thm:JKS} are not possible. Suppose first that $2\nmid a$, so that $q\ge 3$ and $\W=\delta$. Note then that condition \eqref{JKS:C2} of Theorem \ref{Thm:JKS} is not applicable since $q\ge 3$. Finally, we see that conditions \eqref{JKS:C1} and \eqref{JKS:C5} of Theorem \ref{Thm:JKS} are satisfied if and only if $\W$ is squarefree. Suppose next that $2\mid a$, and let $q=2$.  If $2\mid a$ and $2\mid b$, then condition \eqref{JKS:C1} is satisfied if and only if 
   \[(a \mmod{4},\ b\mmod{4})\in \{(0,2),(2,2)\}.\] Suppose then that $2\mid a$ and $2\nmid b$. Then, we have from condition \eqref{JKS:C2} of Theorem \ref{Thm:JKS} that $j=1$, and 
   \[b_1=(b+b^2)/2\equiv \left\{\begin{array}{cl}
   1 \pmod{2} & \mbox{if $b\equiv 1\pmod{4}$,}\\[.5em]
   0 \pmod{2} & \mbox{if $b\equiv 3\pmod{4}$.}
   \end{array}\right.\] If $2\mid a_2$, then $2\nmid {\rm index}(g_2)$ 
   if and only if $b\equiv 1 \pmod{4}$. That is, in this case we have 
   \[(a \mmod{4},\ b\mmod{4})=(0,1).\] 
   If $2\nmid a_2$, then 
   \[a_2^2b+b_1^2\equiv \left\{\begin{array}{cl}
   0 \pmod{2} & \mbox{if $b\equiv 1\pmod{4}$,}\\[.5em]
   1 \pmod{2} & \mbox{if $b\equiv 3\pmod{4}$.}
   \end{array}\right.\]
   Thus, in this case, $2\nmid {\rm index}(g_2)$ if and only if $b\equiv 3 \pmod{4}$, or
   \[(a \mmod{4},\ b\mmod{4})=(2,3).\]
    Finally, as in the case when $2\nmid a$, we deduce from conditions \eqref{JKS:C1} and \eqref{JKS:C5} of Theorem \ref{Thm:JKS} that $q\nmid {\rm index}(g_2)$ if and only if $\W$ is squarefree when $q\ge 3$. 
\end{proof}

The next theorem targets power-compositional situations and is due to Kaur, Kumar and Remete \cite{KKR}. 
 \begin{thm}\label{Thm:KKR}
   Let ${\mathfrak g}(x)$ be a monic polynomial with integer coefficients. Let $k\ge 2$ be an integer such that ${\mathfrak f}(x):={\mathfrak g}(x^k)$ is irreducible over $\Q$. 
   Then ${\mathfrak f}(x)$ is monogenic if and only if all of the following conditions are satisfied: 
   \begin{enumerate}
   \item \label{KKR:I1} ${\mathfrak g}(0)$ is squarefree,
    \item \label{KKR:I2} $q$ does not divide the index of ${\mathfrak f}(x)$ for all primes $q\mid k$,
     \item \label{KKR:I3} ${\mathfrak g}(x)$ is monogenic. 
   \end{enumerate}
 \end{thm}
\noindent 
 We then have the following immediate corollary of Theorem \ref{Thm:KKR}.
\begin{cor}\label{Cor:KKR}
 Suppose that ${\mathfrak F}(x)=x^{2m}+Ax^m+B\in \Z[x]$ is monogenic, where $AB\ne 0$ and $m\ge 2$.  Then $B$ and $(A^2-4B)/\gcd(2,A)^2$ are squarefree. Moreover, ${\mathfrak F}_d(x):=x^{2d}+Ax^d+B$ is monogenic for every divisor $d \ge 1$ of $m$. 
\end{cor}
\noindent The special cases $m\in \{2,3,6\}$ of Corollary \ref{Cor:KKR} will be useful to us in this article.  

We turn now to $g_4(x)$. The following theorem  follows from \cite{KW}, and gives the characterization of the possible Galois groups $G_4$, using the notation 4T$n$ \cite{BM}, in terms of conditions on the coefficients $a$ and $b$. 
\begin{thm}\label{Thm:KW}
        If $g_4(x)$ is irreducible over $\Q$, then $G_4\simeq$
        \begin{enumerate}
          \item {\rm 4T1} if and only if $b\delta$ is a square,
          \item {\rm 4T2} if and only if $b$ is a square,
          \item {\rm 4T3} if and only if neither $b$ nor $b\delta$ is a square.
        \end{enumerate}  
 \end{thm}

 The complete characterization of the 4T1-monogenic trinomials $g_4(x)$ and a partial characterization of the 4T3-monogenic trinomials $g_4(x)$ were given in \cite{JonesQuarticsBAMS}, while a partial characterization of the 4T2-monogenic trinomials $g_4(x)$ was given in \cite{HJBAMS}. By completing the characterizations for the cases 4T2 and 4T3, we present in the following theorem the comprehensive characterization of the monogenic trinomials $g_4(x)$, according to their Galois groups.
\begin{thm}\label{Thm:Main g_4} Suppose that $g_4(x)$ is irreducible over $\Q$. Then $g_4(x)$ is monogenic with $G_4\simeq$
\begin{enumerate}
   \item \label{g4:1} {\rm 4T1} if and only if $g_4(x)\in \{x^4\pm 4x^2+2,\ x^4-5x^2+5\}$; 
   \item \label{g4:2} {\rm 4T2} if and only if $\W$ is squarefree and 
   \[g_4(x)\in \FF_0=\{x^4+ax^2+1: a\mmod{4}\in \{0,3\}\};\] 
   \item \label{g4:3} {\rm 4T3} if and only if $\W$ and $b\ne 1$ are squarefree, $b\delta$ is not a square and 
   \[g_4(x)\in \FF_1=\{x^4+ax^2+b: (a\mmod{4}, \ b\mmod{4})\in \RR\},\] where
   $\RR=\{(0,1),(0,2),(2,2),(2,3),(1,3),(3,2),(3,1),(3,3)\}$.
   \end{enumerate}
\end{thm}
\begin{proof}
 The proof of item \eqref{g4:1} can be found in \cite{JonesQuarticsBAMS}. 
 
 For item \eqref{g4:2}, suppose that $G_4\simeq$ 4T2. If $g_4(x)$ is monogenic, it follows from Corollary \ref{Cor:KKR} that  $g_2(x)$ is monogenic, and both $b$ and $\W$ are squarefree. Since $g_2(x)$ is monogenic, we see from Theorem \ref{Thm:Main g_2} that   
  \begin{equation}\label{g_2 conditions}
 (a \mmod{4}, \ b\mmod{4})\in \{(0,1),(0,2),(2,2),(2,3)\} \quad \mbox{if $2\mid a$.}
 \end{equation}  From Theorem \ref{Thm:KW} we have that $b$ is a square, while from Theorem \ref{Thm:KKR}, we have that $g_2(0)=b$ is squarefree. Hence, $b=1$. Thus, if $2\mid a$, we conclude from \eqref{g_2 conditions} that $a\equiv 0 \pmod{4}$. So, suppose that $2\nmid a$. From Theorem \ref{Thm:KKR}, we only have to check divisibility of the index($g_4$) by $q=2$. Since $g_4(x)$ is monogenic, we have by condition \eqref{JKS:C4} of Theorem \ref{Thm:JKS} with ${\mathfrak F}(x)=g_4(x)$ and $q=2$, that
 \[H_1(x)=x^{2}+ax+1 \ \mbox{ and }\ H_2(x)=\dfrac{ax^{2}+1+\left(-ax-1\right)^{2}}{2}=\left(\frac{a^2+a}{2}\right)x^2+ax+1\]  are coprime in $\F_2[x]$. Since $2\nmid a$, $H_1(x)$ is irreducible in $\F_2[x]$, which implies that $a\equiv 3\pmod{4}$. 
 
 Conversely, suppose that $g_4(x)\in \FF_0$, where $\W$ is squarefree. Then $g_2(x)$ is monogenic from Theorem \ref{Thm:Main g_2}, and we deduce that $g_4(x)$ is monogenic from Theorem \ref{Thm:KKR}.

 For item \eqref{g4:3}, suppose that $G_4\simeq$ 4T3. If $g_4(x)$ is monogenic, then we have that $b$ and $\W$ are both squarefree, and $g_2(x)$ is monogenic, from Corollary \ref{Cor:KKR}. From Theorem \ref{Thm:KW}, we have that $b\delta$ and $b$ are not squares, so that $b\ne 1$.  Moreover, we have from Theorem \ref{Thm:Main g_2} that conditions \eqref{g_2 conditions} hold. Suppose then that $2\nmid a$. From condition \eqref{JKS:C4} of Theorem \ref{Thm:JKS} with $q=2$, we get that
 \begin{align}\label{H1 and H2 for g4} 
 \begin{split}
 H_1(x)&=x^{2}+ax+b \ \mbox{ and }\\
 H_2(x)&=\dfrac{ax^{2}+b+\left(-ax-b\right)^{2}}{2}=\left(\frac{a^2+a}{2}\right)x^2+abx+\frac{b^2+b}{2}
 \end{split}
 \end{align}
 are coprime in $\F_2[x]$. If $2\mid b$, then $b\equiv 2\pmod{4}$ since $b$ is squarefree. Then
 \[H_1(x)\equiv x(x+a) \pmod{2} \quad \mbox{and} \quad H_2(x)\equiv\left(\frac{a^2+a}{2}\right)x^2+1 \pmod{2}\] from \eqref{H1 and H2 for g4}, which implies that $a\equiv 3 \pmod{4}$. If $2\nmid b$, then we have from \eqref{H1 and H2 for g4} that $H_1(x)\equiv x^2+x+1\pmod{2}$ is irreducible in $\F_2[x]$. Thus, the only way that $H_1(x)$ and $H_2(x)$ would not be coprime in $\F_2[x]$ is if $H_2(x)\equiv H_1(x)\pmod{2}$, which implies that $(a \mmod{4},\ b\mmod{4})=(1,1)$. Hence, we must have 
 \[(a \mmod{4},\ b\mmod{4})\in \{(1,3),(3,1),(3,3)\}.\] 

The converse is similar to the previous case 4T2, and we omit the details. 
 \end{proof}

We turn next to $g_6(x)$. The characterization of the monogenic trinomials $g_6(x)$ according to their Galois groups is given below \cite{HJActa}.
\begin{thm}\label{Thm:Main g_6}  
If $g_6(x)$ is irreducible over $\Q$, then
$g_6(x)$ is monogenic with $G_6\simeq$
\begin{enumerate}
  \item \label{g6:I1} {\rm 6T1} if and only if $g_6(x)\in \{x^6-x^3+1,\ x^6+x^3+1\}$;
  \item \label{g6:I2} {\rm 6T2} never occurs;
 \item \label{g6:I3} {\rm 6T3} if and only if $g_6(x)\in \FF_i$ for some  $i\in \{2,3,4\}$, where
 \begin{enumerate}
   \item $\FF_2:=\{x^6-2x^3+2,\ x^6+2x^3+2\}$,
   \item $\FF_3:=\{x^6+ax^3+1: a \mmod{9}\ne 0,\ a\ne \pm 1,$\\
    $\hspace*{1in} \mbox{with} \ a-2 \ \mbox{and}\ a+2 \ \mbox{squarefree}\}$,
   \item $\FF_4:=\left\{x^6+ax^3-1: a \mmod{4}\ne 0, \ a \mmod{9}\not \in \{0,4,5\},\right.$\\
   $\hspace*{1in}\left. \mbox{with} \ (a^2+4)/\gcd(a^2+4,4) \ \mbox{squarefree} \right\}$;
    \end{enumerate}
  \item \label{g6:I4} {\rm 6T5} if and only if $g_6(x)\in \FF_5$, where
    \begin{align*}
  \FF_5:=&\{x^6+ax^3+(a^2+3)/4: a \mmod{2}=1,\ a\ne \pm 1, \\
  &\hspace*{1.5in} \mbox{with} \ (a^2+3)/4 \ \mbox{squarefree}\};
  \end{align*}
   \item \label{g6:I5} {\rm 6T9} if and only if $g_6(x)$ is contained in one of $144$ infinite pairwise-disjoint $2$-parameter monogenic families. 
   \end{enumerate} 
\end{thm}

 Let $P$, $Q$, $R$ and $S$ denote the following statements: 
   \begin{align}\label{PQRS}
   \begin{split}
     P:& \quad -3b\delta \ \mbox{is a square},\\ 
     Q:& \quad -3b \ \mbox{is a square},\\
     R:& \quad  r(x) \ \mbox{is reducible},\\
     S:& \quad b \ \mbox{is a cube}.
     \end{split}
   \end{align}  
   Making use of \eqref{PQRS}, the next theorem uses the possible pairs $(G_4,G_6)$ and conditions on the coefficients $a$ and $b$ of $f(x)$ to determine $\Gal(f)$.
\begin{thm}{\rm \cite{Chen}}\label{Thm:Chen}
Recall the definitions as given in \eqref{Basics} and \eqref{PQRS}.   
The pairs $(G_4,G_6)\in \{(4{\rm T}1,6{\rm T}1),(4{\rm T}1,6{\rm T}2),(4{\rm T}1,6{\rm T}5)\}$ are not possible. Table \ref{T:Chen} indicates how the pair $(G_4,G_6)$ determines the possibilities for $\Gal(f)$.  
\begin{table}[h]
 \begin{center}
\begin{tabular}{c|ccc}
$(G_4,G_6)$ & \multicolumn{3}{c}{$\Gal(f)$} \\ \hline
$(4{\rm T}1,6{\rm T}3)$ & $12{\rm T}11$\\
$(4{\rm T}1,6{\rm T}9)$ & $12{\rm T}39$\\
$(4{\rm T}2,6{\rm T}1)$ & $12{\rm T}2$\\
$(4{\rm T}2,6{\rm T}2)$ & $12{\rm T}3$\\
$(4{\rm T}2,6{\rm T}5)$ & $12{\rm T}18$\\
$(4{\rm T}2,6{\rm T}3)$ & $12{\rm T}3$ & $12{\rm T}10$\\
& \mbox{if $3\alpha$ or $3\beta$} & \mbox{otherwise}\\
& \mbox{is a square} \\ 
$(4{\rm T}2,6{\rm T}9)$ & $12{\rm T}16$ & $12{\rm T}37$\\
& \mbox{if $3\alpha$ or $3\beta$} & \mbox{otherwise}\\
& \mbox{is a square} \\ 
$(4{\rm T}3,6{\rm T}1)$ & $12{\rm T}14$\\
$(4{\rm T}3,6{\rm T}2)$ & $12{\rm T}15$\\
$(4{\rm T}3,6{\rm T}5)$ & $12{\rm T}42$\\
$(4{\rm T}3,6{\rm T}3)$ & $12{\rm T}12$ & $12{\rm T}13$ & $12{\rm T}28$\\
& \mbox{if $P$ and $R$ are true} & \mbox{if $Q$ and $R$ are true}& \mbox{otherwise}\\
& \mbox{or $Q$ and $S$ are true} & \mbox{or $P$ and $S$ are true}\\
$(4{\rm T}3,6{\rm T}9)$ & $12{\rm T}38$ & $12{\rm T}81$\\
& \mbox{if $-3b$ or $-3b\delta$} & \mbox{otherwise}\\
& \mbox{is a square} 
\end{tabular}
\end{center}
\caption{$\Gal(f)$ from possible pairs $(G_4,G_6)$}
 \label{T:Chen}
\end{table}  
\end{thm}

\section{Proof of Theorem \ref{Thm:Main}}
\begin{proof}
  For each possible pair $(G_4,G_6)$ yielding 12T$n$ from Theorem \ref{Thm:Chen}, we use Theorem \ref{Thm:Main g_4}, Theorem \ref{Thm:Main g_6} and Theorem \ref{Thm:KKR} 
  to determine the 12T$n$-monogenic trinomials $f(x)$. More precisely, if $f(x)=x^{12}+ax^6+b$ with $\Gal(f)\simeq$ 12T$n$, which arises from $(G_4,G_6)$ in Theorem \ref{Thm:Chen}, then it follows from Theorem \ref{Thm:KKR} that $f(x)$ is monogenic if and only if both $g_4(x)$ and $g_6(x)$ are monogenic, $b$ is squarefree and $\gcd({\rm index}(f),6)=1$. Consequently, we need to determine the pairs $(a,b)$ that occur for both $g_4(x)$ and $g_6(x)$ in, respectively, Theorem \ref{Thm:Main g_4} and Theorem \ref{Thm:Main g_6} for the specific pair $(G_4,G_6)$. Recall, from Corollary \ref{Cor:KKR},  that $b$ and $\W$ are squarefree for every possible monogenic $g_4(x)$ and $g_6(x)$. Therefore, once we have determined the viable monogenic candidates $f(x)$ from this process, we only need to check that $\gcd({\rm index}(f),6)=1$, and we can do so using Theorem \ref{Thm:JKS}. One important item that follows from Theorem \ref{Thm:KKR} is that if both $g_4(x)$ and $g_6(x)$ are known to be monogenic, then we can determine the monogenicity of $f(x)$ by examining either $q=2$ (since $f(x)=g_6(x^2)$) or $q=3$ (since $f(x)=g_4(x^3)$) in Theorem \ref{Thm:JKS}. That is,  Theorem \ref{Thm:KKR} does not require the examination of both $q=2$ and $q=3$ in Theorem \ref{Thm:JKS} to determine the monogenicity of $f(x)$. Therefore, we can choose to examine $q=2$ or $q=3$, depending on which situation requires fewer calculations. 
     
\subsection*{{\bf The Case} $\mathbf{12T2}$} If $\Gal(f)\simeq$ 12T2, then we have from Theorem \ref{Thm:Chen} that $(G_4,G_6)=(4{\rm T}2,6{\rm T}1)$. Observe that $(a,b)=(-1,1)$ is the only pair for which $g_4(x)$ and $g_6(x)$ are both monogenic with the appropriate $(G_4,G_6)$ from Theorem \ref{Thm:Main g_4} and Theorem \ref{Thm:Main g_6}. Thus, $f(x)=x^{12}-x^6+1=\Phi_{36}(x)$ (the cyclotomic polynomial of index 36), which is well known to be monogenic \cite{Washington}, is the only monogenic trinomial $f(x)$ with $\Gal(f)\simeq$ 12T2. Knowing that $f(x)=\Phi_{36}(x)$ is irreducible, we can verify independently the fact that $f(x)=g_6(x^2)$ is monogenic by appealing to Theorem \ref{Thm:KKR} and examining condition \eqref{JKS:C4} of Theorem \ref{Thm:JKS} with $q=2$. We get
\begin{align*}
H_1(x)&=x^6-x^3+1\equiv x^6+x^3+1 \pmod{2} \quad \mbox{and}\\
H_2(x)&=\frac{-x^6+1+(x^3-1)^2}{2}\equiv (x+1)(x^2+x+1) \pmod{2},
\end{align*} 
from which it is easily seen that $H_1(x)$ and $H_2(x)$ are coprime in $\FF_2[x]$. It follows that $2\nmid $ index($f$), and we conclude that $f(x)$ is monogenic by Theorem \ref{Thm:KKR}. 

\subsection*{{\bf The Case} $\mathbf{12T3}$} If $\Gal(f)\simeq$ 12T3, then we see from Theorem \ref{Thm:Chen} that there are two possibilities in this case. However, the possibility $(G_4,G_6)=(4{\rm T}2,6{\rm T}2)$ cannot yield any monogenic trinomials $f(x)$ since there are no 6T2-monogenic trinomials $g_6(x)$ by Theorem \ref{Thm:Main g_6}. So, suppose that $(G_4,G_6)=(4{\rm T}2,6{\rm T}3)$, where 
\begin{equation}\label{12T3: 3alpha or 3beta is a square} 
3\alpha \quad \mbox{or} \quad 3\beta \quad \mbox{is a square.}
\end{equation} 
If $f(x)$ is monogenic, then, since $g_4(x)$ must be monogenic by Theorem \ref{Thm:KKR}, we see from Theorem \ref{Thm:Main g_4} that $b=1$. Hence,  
\begin{equation}\label{12T3: 3alpha and 3beta}
3\alpha=3(a+2) \quad \mbox{and} \quad 3\beta=3(a-2).
\end{equation} Furthermore, since $b=1$, and $g_6(x)$ must be monogenic by Theorem \ref{Thm:KKR}, we see from Theorem \ref{Thm:Main g_6} that $g_6(x)\in \FF_3$ with both $a+2$ and $a-2$ squarefree. It follows from \eqref{12T3: 3alpha or 3beta is a square} and \eqref{12T3: 3alpha and 3beta} that $a\in \{1,5\}$, which contradicts the fact that $a \mmod{4}\in \{0,3\}$ by Theorem \ref{Thm:Main g_4}. Thus, there are no 12T3-monogenic trinomials $f(x)$.

\subsection*{{\bf The Case} $\mathbf{12T10}$} Suppose that $\Gal(f)\simeq$ 12T10. Then we have from Theorem \ref{Thm:Chen} that $(G_4,G_6)=(4{\rm T}2,6{\rm T}3)$, such that neither $3\alpha$ nor $3\beta$ is a square. 

Suppose that $f(x)$ is monogenic. 
We see from the analysis given in Case 12T3 that $b=1$, $a\not \in \{1,5\}$ and $a\mmod 4\in \{0,3\}$. Furthermore, by Theorem \ref{Thm:Main g_6}, we have that $a\ne \pm 1$ and $a\mmod {9}\ne 0$, with $a+2$ and $a-2$ both squarefree, which is equivalent to $\W$ being squarefree. Thus, $f(x)$ satisfies conditions $\CC_{10}$.

Conversely, suppose that $f(x)=x^{12}+ax^6+1$ satisfies conditions $\CC_{10}$. Then, from Theorem \ref{Thm:Main g_6}, we see that $g_6(x)$ is monogenic. Since $f(x)=g_6(x^2)$, it follows from Theorem \ref{Thm:KKR} that we can determine the monogenicity of $f(x)$ by examining the prime $q=2$ in Theorem \ref{Thm:JKS}. If $2\mid a$, then $4\mid a$ from conditions $\CC_{10}$, and in this case we see that $q=2$ satisfies condition \eqref{JKS:C2} of Theorem \ref{Thm:JKS} since $2\mid a_2$ and $b_1=1$. Thus, $2\nmid $ index($f$) and $f(x)$ is monogenic. If $2\nmid a$, then we must examine condition \eqref{JKS:C4} of Theorem \ref{Thm:JKS}. Note that $a\mmod 4=3$ from conditions $\CC_{10}$. We get 
\[H_1(x)=x^6+ax^3+1\equiv x^6+x^3+1 \pmod{2},\]
which is irreducible in $\F_2[x]$, and 
\begin{align*}
H_2(x)&=\frac{ax^6+1+(-ax^3-1)^2}{2}\\
&=\left(\frac{a^2+a}{2}\right)x^6+ax^3+1\\
&\equiv (x+1)(x^2+x+1) \pmod{2}.
\end{align*} It is then easy to see that $H_1(x)$ and $H_2(x)$ are coprime in $\F_2[x]$, so that $2\nmid$ index($f$) by Theorem \ref{Thm:JKS}, and $f(x)$ is monogenic. 

\subsection*{{\bf The Case} $\mathbf{12T11}$} We see from Theorem \ref{Thm:Chen} that $(G_4,G_6)=(4{\rm T}1,6{\rm T}3)$.
  By inspection of Theorem \ref{Thm:Main g_4} and Theorem \ref{Thm:Main g_6}, we conclude that there do not exist any monogenic trinomials $f(x)$ with $\Gal(f)\simeq$ 12T11. 
\subsection*{{\bf The Case} $\mathbf{12T12}$} Suppose that $\Gal(f)\simeq$ 12T12. Then, from Theorem \ref{Thm:Chen}, we have that $(G_4,G_6)=(4{\rm T}3,6{\rm T}3)$, such that 
either $-3b\delta$ is a square and $r(x)$ is reducible, or $-3b$ is a square and $b$ is a cube. 

Suppose that $f(x)$ is monogenic. Then, from Corollary \ref{Cor:KKR}), we have that $b$ and $\W$ are squarefree, and $g_2(x)$, $g_4(x)$ and $g_6(x)$ are all monogenic.

 Assume first that $-3b\delta$ is a square and $r(x)$ is reducible over $\Q$. Then $3\mid \mid b\delta$, and $-b(a^2-4b)/3$ is a square. Suppose that $2\nmid a$. If $3\mid b$, then since $b/3$ and $a^2-4b$ are squarefree, it follows that $-b/3=a^2-4b$, or equivalently, $a^2=11b/3$. Hence, $(a,b)\in \{(\pm 11,33)\}$, which contradicts Theorem \ref{Thm:Main g_6}. If $3\mid \delta$, then $-b=\delta/3$, or equivalently, $a^2=b$. Thus, $(a,b)\in \{(\pm 1,1)\}$. It is easy to see that the pair $(a,b)=(1,1)$ contradicts Theorem \ref{Thm:Main g_4}, while a straightforward check reveals that the pair $(a,b)=(-1,1)$ yields the monogenic trinomial $f(x)=x^{12}-x^6+1$. However, in this case, 
 \[r(x)=x^3-3bx+ab=x^3-3x-1\] is irreducible over $\Q$, which contradicts the fact that $\Gal(f)\simeq$ 12T12. (Indeed, we have already shown that $\Gal(f)\simeq$ 12T2 in this case.) Suppose then that $2\mid a$. If $3\mid b$, then $(-b/3)((a/2)^2-b)$ is a square. Since $-b/3$ and $(a/2)^2-b$ are squarfree, we have that $-b/3=(a/2)^2-b$, or equivalently, $(a/2)^2=2(b/3)$. Therefore, $(a,b)\in \{(\pm 4,6)\}$, which contradicts Theorem \ref{Thm:Main g_6}. If $3\mid \delta$, then $-b=((a/2)^2-b)/3$, or equivalently, $(a/2)^2=-2b$. Hence, $(a,b)\in \{(\pm 4,-2)\}$, which again contradicts Theorem \ref{Thm:Main g_6}. 
 
  Assume next that $-3b$ is a square and $b$ is a cube. Since $b$ is squarefree and $b$ is a cube, we have that $b=\pm 1$, which contradicts the fact that $-3b$ is a square. Hence, this final contradiction confirms that there do not exist any monogenic trinomials $f(x)$ with $\Gal(f)\simeq$ 12T12. 
  
  \subsection*{{\bf The Case} $\mathbf{12T13}$} Suppose that $\Gal(f)\simeq$ 12T13. From Theorem \ref{Thm:Chen}, we have that $(G_4,G_6)=(4{\rm T}3,6{\rm T}3)$, such that either $-3b$ is a square and $r(x)$ is reducible, or $-3b\delta$ is a square and $b$ is a cube.
   
   Suppose that $f(x)$ is monogenic. From Corollary \ref{Cor:KKR}), we have that $g_2(x)$, $g_4(x)$ and $g_6(x)$ are all monogenic, with $b$ and $\W$ squarefree. 
  
  Assume first that $-3b$ is a square and $r(x)$ is reducible over $\Q$. Since $b$ is squarefree, it follows that $b=-3$, which contradicts Theorem \ref{Thm:Main g_6}. 
  
  Assume next that $-3b\delta$ is a square and $b$ is a cube. The same arguments used in the case 12T12 yield a contradiction in every possible situation, and we omit the details.  Thus, there exist no 12T13-monogenic trinomials $f(x)$. 
  
  \subsection*{{\bf The Case} $\mathbf{12T14}$} Suppose that $\Gal(f)\simeq$ 12T14. From Theorem \ref{Thm:Chen}, we have that $(G_4,G_6)=(4{\rm T}3,6{\rm T}1)$. By inspection, we deduce from Theorem \ref{Thm:Main g_4} and Theorem \ref{Thm:Main g_6} that $(a,b)=(-1,1)$, so that $f(x)=x^{12}-x+1$. However, we have already shown that $\Gal(f)=$ 12T2. Hence, there are no monogenic trinomials $f(x)$ in this case.  
  
  \subsection*{{\bf The Case} $\mathbf{12T15}$} Suppose that $\Gal(f)\simeq$ 12T15. From Theorem \ref{Thm:Chen}, we have that $(G_4,G_6)=(4{\rm T}3,6{\rm T}2)$. However, there are no 6T2-monogenic trinomials $g_6(x)$ by Theorem \ref{Thm:Main g_6}, and therefore, there are no 12T15-monogenic trinomials $f(x)$ by Theorem \ref{Thm:KKR}. 
  
  \subsection*{{\bf The Case} $\mathbf{12T16}$} Suppose that $\Gal(f)\simeq$ 12T16. From Theorem \ref{Thm:Chen}, we have that $(G_4,G_6)=(4{\rm T}2,6{\rm T}9)$, such that either $3\alpha$ or $3\beta$ is a square. 
  
  If $f(x)$ is monogenic, then  $g_4(x)$ is monogenic by Corollary \ref{Cor:KKR}. Then $a \mmod{4}\in \{0,3\}$ and $b=1$ from Theorem \ref{Thm:Main g_4}. Hence, 
  \[3\alpha \mmod{4}=3(a+2) \mmod{4} \in \{2,3\} \quad \mbox{and} \quad 3\beta \mmod{4}=3(a-2) \mmod{4}\in \{2,3\},\]
  so that neither $3\alpha$ nor $3\beta$ is a square. This contradiction shows that no 12T16-monogenic trinomials $f(x)$ exist. 
  
   \subsection*{{\bf The Case} $\mathbf{12T18}$} Suppose that $\Gal(f)\simeq$ 12T18. From Theorem \ref{Thm:Chen}, we have that $(G_4,G_6)=(4{\rm T}2,6{\rm T}5)$. 
   
   If $f(x)$ is monogenic, then $g_4(x)$ and $g_6(x)$ are monogenic by Corollary \ref{Cor:KKR}. Hence, $b=1=(a^2+3)/4$ from Theorem \ref{Thm:Main g_4} and Theorem \ref{Thm:Main g_6}. Thus, $a=\pm 1$, which contradicts Theorem \ref{Thm:Main g_6}. Hence, there exist no 12T18-monogenic trinomials $f(x)$.    
   
   \subsection*{{\bf The Case} $\mathbf{12T28}$} From Theorem \ref{Thm:Chen}, we have that $(G_4,G_6)=(4{\rm T}3,6{\rm T}3)$. Furthermore, to ensure that $\Gal(f)\simeq$ 12T28, we must have from Theorem \ref{Thm:Chen} that the negation of the statement 
   \[Y:=[(P\land R) \lor (Q\land S)]\lor [(Q\land R) \lor (P\land S)]\] is true, where $P$, $Q$, $R$ and $S$ are the statements as given in \eqref{PQRS}. The seven possibilities for the truth values of the statements $P$, $Q$, $R$ and $S$ for which $\neg Y$ is true are given in Table \ref{T:28}. 
   \begin{table}[h]
 \begin{center}
\begin{tabular}{c|ccccccc}
$P$ & T & T & F & F & F & F & F   \\ 
$Q$ & T & F & T & F & F & F & F  \\
$R$ & F & F & F & T & T & F & F  \\
$S$ & F & F & F & T & F & T & F
\end{tabular}
\end{center}
\caption{Truth values of $P$, $Q$, $R$ and $S$ for which $\neg Y$ is true}
 \label{T:28}
\end{table} 

Suppose that $f(x)$ is monogenic. Then $g_2(x)$, $g_4(x)$ and $g_6(x)$ are all monogenic, with $b$ and $\W$ squarefree, from Corollary \ref{Cor:KKR}). 
We use Theorem \ref{Thm:Main g_4} and Theorem \ref{Thm:Main g_6} with $(G_4,G_6)=(4{\rm T}3,6{\rm T}3)$ to impose further restrictions on statements $P$, $Q$, $R$ and $S$ to determine the possibilities in Table \ref{T:28} that could actually yield monogenic trinomials $f(x)$, and to derive necessary conditions for the monogenicity of $f(x)$. 

 Observe that $b\ne 1$ from Theorem \ref{Thm:Main g_4} and $b\in \{\pm 1,2\}$ from Theorem \ref{Thm:Main g_6}. Thus, $b\in \{-1,2\}$, and 
 \[-3b\delta=\left\{\begin{array}{cl}
   3(a^2+4) & \mbox{if $b=-1$},\\
   -6(a^2-8)  & \mbox{if $b=2$.}
 \end{array}
 \right.\]
If $3(a^2+4)$ is a square, then $3\mid (a^2+4)$ and we arrive at the contradiction that  $a^2\equiv 2 \pmod{3}$. The same contradiction is reached if we assume that $-6(a^2-8)$ is a square. Hence, statement $P$ is false, which eliminates the first two possibilities in Table \ref{T:28}. A similar argument, which we omit, shows that statement $Q$ is false, eliminating the third possibility in Table \ref{T:28}.
  
If $b=2$, then statement $S$ is false, and $a=\pm 2$ from Theorem \ref{Thm:Main g_6}. Note also that 
\[r(x)=x^3-6x\pm 4=(x\mp2)(x^2\pm2x-2),\] so that statement $R$ is true. Thus, we get that $f(x)=x^{12}\pm 2x^6+2$.

If $b=-1$, then statement $S$ is true. Moreover, for $a\ne 0$, it is easy to see that $r(x)$ is reducible if and only if $a=n^3+3n$ for $n\in \Z$, and in this case we have
\[r(x)=(x-n)(x^2+nx+n^2+3),\] which implies that 
\begin{align*}
f(x)&=x^{12}+(n^3+n)x^6-1\\
&=(x^4+nx^2-1)(x^8-nx^6+n^2x^4+x^4+nx^2n+1).
\end{align*} Thus, we can assume that $r(x)$ is irreducible so that statement $R$ is false.  
In this situation, we have that $g_4(x)\in \FF_1$ from Theorem \ref{Thm:Main g_4} and $g_6(x)\in \FF_4$ from Theorem \ref{Thm:Main g_6}, such that $a\mmod{4}\ne 0$ and $a\mmod{9}\not \in \{0,4,5\}$. That is, we have shown that if $f(x)$ is monogenic, then $f(x)\in \{x^{12}+2x^6+2,\ x^{12}-2x^6+2\}$ or $f(x)=x^{12}+ax^6-1$, where $a$ satisfies the conditions $\CC_{28}$. We have also established that the last possibility in Table \ref{T:28} cannot occur when $f(x)$ is monogenic. 

Conversely, suppose that $f(x)=x^{12}\pm 2x^6+2$ or $f(x)=x^{12}+ax^6-1$, where $a$ satisfies the conditions $\CC_{28}$. It is easy to see that each of the irreducible (2-Eisenstein) trinomials  $f(x)=x^{12}+2x^6+2$ and  $f(x)=x^{12}-2x^6+2$ satisfies condition \eqref{JKS:C1} of Theorem \ref{Thm:JKS} with $q=2$, which implies that they are both monogenic by Theorem \ref{Thm:KKR} since $g_6(x)$ is monogenic in each case by Theorem \ref{Thm:Main g_6}, and $f(x)=g_6(x^2)$. 

Suppose then that $f(x)=x^{12}+ax^6-1$, where $a$ satisfies the conditions $\CC_{28}$. We use Theorem \ref{Thm:JKS} to examine the monogenicity of $f(x)=x^{12}+ax^6-1$. Suppose first that $q=2$ divides $a$. Then, by condition \eqref{JKS:C2} of Theorem \ref{Thm:JKS}, $2\nmid$ index($f$) since $4\nmid a$. If $2\nmid a$, then in condition \eqref{JKS:C4} of Theorem \ref{Thm:JKS}, we have that 
\[H_1(x)=x^6+ax^3+1 \quad \mbox{and}\quad H_2(x)=x^3\left(\left(\frac{a^2+a}{2}\right)x^3-a\right).\] Since $2\nmid a$, we have that $H_1(x)$ is irreducible in $\F_2[x]$, which implies that $H_1(x)$ and $H_2(x)$ are coprime in $\F_2[x]$. Thus, $2\nmid$ index($f$). Suppose next that $q=3$ is a divisor of $a$. Then, we see from condition \eqref{JKS:C2} of Theorem \ref{Thm:JKS} that $3\nmid$ index($f$) since $a\not \equiv 0 \pmod{9}$. If $3\nmid a$, then we see by inspection of Table \ref{T:H1 H2 q=3}, generated from condition \eqref{JKS:C4} of Theorem \ref{Thm:JKS} with $a \mmod{9}\in \{1,2,7,8\}$, that $3\nmid$ index($f$), which verifies that $f(x)$ is monogenic. 
  \begin{table}[h]
 \begin{center}
\begin{tabular}{ccc}
$a \mmod{9}$ & $H_1(x) \mmod{3}$ & $H_2(x)\mmod{3}$   \\ \hline
1 & $x^4+x^2+2$ & $x^2(x+1)(x+2)$ \\
2 & $x^4+2x^2+2$ & $x^2(x+1)^2(x+2)^2$ \\
7 & $x^4+x^2+2$ & $2x^2(x^2+1)^2$ \\
8 & $x^4+2x^2+2$ & $x^2(x^2+1)$.
\end{tabular}
\end{center}
\caption{Factorization of $H_i(x)$ into irreducibles in $\F_3[x]$}
 \label{T:H1 H2 q=3}
\end{table}\\

 \subsection*{{\bf The Case} $\mathbf{12T37}$} Suppose that $\Gal(f)\simeq$ 12T37. 
   From Theorem \ref{Thm:Chen}, we have that $(G_4,G_6)=(4{\rm T}2,6{\rm T}9)$, where neither $3\alpha$ nor $3\beta$ is a square. 
   
   Suppose that $f(x)$ is monogenic. Then we have from Corollary \ref{Cor:KKR}) that $g_2(x)$, $g_4(x)$ and $g_6(x)$ are all monogenic, with $b$ and $\W$ squarefree. From Theorem \ref{Thm:Main g_4}, we see that $a\mmod{4}\in \{0,3\}$ and $b=1$. Consequently, both $a-2$ and $a+2$ are squarefree since $\W$ is squarefree. Note also that if $a=-1$, then $f(x)=x^{12}-x^6+1$ so that $\Gal(f)\simeq$ 12T2. Hence, $a\ne \pm 1$. Furthermore, $a\mmod{9}=0$, otherwise we would have $g_6(x)\in \FF_3$, so that $G_6=6{\rm T}3$ by Theorem \ref{Thm:Main g_6}, contradicting the fact that $G_6=6{\rm T}9$. 
   
   Since 3 divides $\Delta(f)=2^{12}3^{12}(a-2)^6(a+2)^6$,  $a\mmod{9}=0$ and $b=1$, we know that condition \eqref{JKS:C2} of Theorem \ref{Thm:JKS} must be satisfied with $q=3$. However, $3\mid a_2$ and $b_1=0$ in condition \eqref{JKS:C2} of Theorem \ref{Thm:JKS}, which implies that $3\mid$ index($f$), and contradicts the assumption that $f(x)$ is monogenic. Hence, there do not exist any 12T37-monogenic trinomials $f(x)$. 

 \subsection*{{\bf The Case} $\mathbf{12T38}$} Suppose that $\Gal(f)\simeq$ 12T38. From Theorem \ref{Thm:Chen}, we have that $(G_4,G_6)=(4{\rm T}3,6{\rm T}9)$, where $-3b$ or $-3b\delta$ is a square. If $-3b$ and $-3b\delta$ are both squares, then $\delta$ is a square, which contradicts the fact that $f(x)$ is irreducible. Hence, $-3b$ and $-3b\delta$ cannot both be squares. 
  
  Suppose that $f(x)$ is monogenic. Then we have from Corollary \ref{Cor:KKR}) that $g_2(x)$, $g_4(x)$ and $g_6(x)$ are all monogenic, with $b$ and $\W$ squarefree.
  
  Assume first that $-3b$ is a square. Since $b$ is squarefree, it follows that $b=-3$, so that   $a\mmod{4}\in \{0,3\}$ from Theorem \ref{Thm:Main g_4}. Note then that $b\delta=-3(a^2+12)$ is not a square. We need to determine the precise conditions under which $g_6(x)=x^6+ax^3-3$ is monogenic. To do this, we only need to examine the prime $q=3$ in Theorem \ref{Thm:JKS} by Theorem \ref{Thm:KKR}, since $g_2(x)=x^2+ax-3$ is monogenic and $g_6(x)=g_2(x^3)$. Furthermore, since $g_6(0)=-3$ is squarefree, we only have to analyze condition \eqref{JKS:C3} of Theorem \ref{Thm:JKS}, where we have 
 \[a_1=\frac{a-a^3}{3} \quad \mbox{and} \quad b_2=-1.\] We see that $3\mid a_1$ if and only if $a\mmod{9}\in \{0,1,8\}$, and that 
 \[3\nmid a_1(aa_1+1) \quad \mbox{if and only if} \quad a\mmod{9}\in \{3,4,5,6\}.\] Thus, it follows that $g_6(x)$ and, consequently, $f(x)$ are monogenic if and only if $a \mmod{9}\not \in \{2,7\}$.

 Assume next that $-3b\delta$ is a square. From the arguments used in the case 12T12, we have that 
 \begin{equation}\label{(a,b) 12T38}
 (a,b)\in \{(\pm 1,1),\ (\pm 4,-2),\ (\pm 4,6), \ (\pm 11,33)\}.
 \end{equation} It is easy to verify that $(a,b)=(1,1)$ is impossible since $f(x)$ is reducible in this situation, while $(a,b)=(-1,1)$ gives $f(x)$ with $\Gal(f)\simeq$ 12T2. With $(a,b)=(\pm 11,33)$,  we see that condition \eqref{JKS:C3} of Theorem \ref{Thm:JKS} fails for $g_6(x)$ with $q=3$. Thus, neither one of the two trinomials $f(x)=x^{12}\pm 11x^6+33=g_6(x^2)$ is monogenic by Theorem \ref{Thm:KKR}. Straightforward calculations using Theorem \ref{Thm:JKS}, or a computer algebra system, verify that the remaining four possibilities in \eqref{(a,b) 12T38} yield 12T38-monogenic trinomials $f(x)$. 
 
 Conversely, suppose that $f(x)=x^{12}+ax^6-3$, such that $a$ satisfies the conditions in $\CC_{38}$. Then, 
 \[(a\mmod{4},\ b\mmod{4})\in \{(0,1),\ (3,1)\}\] and $b\delta=-3(a^2+12)<0$ is not a square. With the calculations done in the other direction of this proof, it follows that $g_4(x)$ is monogenic with $G_4=$ 6T3 and $g_6(x)$ is monogenic with $G_6=$ 6T9, so that $f(x)$ is monogenic by Theorem \ref{Thm:KKR}.

 \subsection*{{\bf The Case} $\mathbf{12T39}$} Suppose that $\Gal(f)\simeq$ 12T39. From Theorem \ref{Thm:Chen}, we have that $(G_4,G_6)=(4{\rm T}1,6{\rm T}9)$. If $f(x)$ is monogenic, then $g_4(x)$ is monogenic from Corollary \ref{Cor:KKR}. Hence, 
  \[f(x)\in S:=\{x^{12}\pm 4x^6+2,\ x^{12}-5x^6+5\}\]  from Theorem \ref{Thm:Main g_4}. Then, using Theorem \ref{Thm:JKS} or a computer algebra system, it is straightforward to confirm that each $f(x)\in S$ is monogenic with $\Gal(f)\simeq$ 12T39.

\subsection*{{\bf The Case} $\mathbf{12T42}$} Suppose that $\Gal(f)\simeq$ 12T42. Then, from Theorem \ref{Thm:Chen}, we have that $(G_4,G_6)=(4{\rm T}3,6{\rm T}5)$. 

Suppose that $f(x)$ is monogenic. Then $g_2(x)$, $g_4(x)$ and $g_6(x)$ are all monogenic, with $b$ and $\W$ squarefree, from Corollary \ref{Cor:KKR}). 
We deduce then from Theorem \ref{Thm:Main g_4} and Theorem \ref{Thm:Main g_6} that $b=(a^2+3)/4$ is squarefree and
\[(a \mmod{4},\ b \mmod{4})\in \{(1,3), (3,1),(3,3)\},\] with $a\ne \pm 1$. By Theorem \ref{Thm:KKR}, since $f(x)=g_6(x^2)$, we only need to focus on $q=2$ in Theorem \ref{Thm:JKS} to ascertain further restrictions on $a$ to guarantee that $f(x)$ is monogenic. From condition \eqref{JKS:C4} of Theorem \ref{Thm:JKS} we get
\begin{align*}\label{H1 H2 42}
 H_1(x):&=x^6+ax^3+(a^2+3)/4 \quad \mbox{and}\\ 
 H_2(x):&=\left(\frac{a(a+1)}{2}\right)x^6+\left(\frac{a(a^2+3)}{4}\right)x^3+\frac{(a^2+3)(a^2+7)}{32}. 
\end{align*}  
Since $2\nmid a$ and $2^2\mid\mid (a^2+3)$, we have that $H_1(x)\equiv x^6+x^3+1 \mmod{2}$ is irreducible in $\F_2[x]$. Thus, $2\mid$ index($f$) if and only if $H_2(x)\equiv H_1(x) \mmod{2}$. Note that the coefficient on $x^3$ in both $H_2(x)$ and $H_1(x)$ modulo 2 is always 1, while equating the coefficients on $x^6$ and the constant terms modulo 2 yields 
\begin{equation}\label{cong}
\frac{a^2+a}{2} \equiv 1 \mmod{2} \quad \mbox{and} \quad \frac{a^2+7}{8} \equiv 1\mmod{2}.
\end{equation}  The first congruence in \eqref{cong} is true if and only if $a\mmod 4=1$, while the second congruence is true if and only if $a\mmod 16\in \{1,7,9,15\}$. Hence, for both congruences in \eqref{cong} to hold simultaneously, we must have $a\mmod{8}=1$. We conclude that $2\nmid$ index($f$) if and only if $a\mmod{8} \in \{3,5,7\}$. Hence, $a$ satisfies the conditions in $\CC_{42}$.

Conversely, suppose that $f(x)=x^{12}+ax^6+(a^2+3)/4$, where $a$ satisfies the conditions in $\CC_{42}$. Observe also that $\W=-3$ is squarefree, and $a\ne 1$, since otherwise $f(x)$ is reducible over $\Q$. Furthermore, $b\delta=-3(a^2+3)/4<0$ is not a square.  Thus, $g_4(x)$ is monogenic with $G_4=$ 4T3 and $g_6(x)$ is monogenic with $G_6=$ 6T5. Since $f(x)=g_6(x^2)$, it follows from Theorem \ref{Thm:KKR} that $f(x)$ is monogenic form the calculations done in the previous direction of the proof for condition \ref{JKS:C4} of Theorem \ref{Thm:JKS} with $q=2$.        

 \subsection*{{\bf The Case} $\mathbf{12T81}$} Suppose that $\Gal(f)\simeq$ 12T81. From Theorem \ref{Thm:Chen}, we have that $(G_4,G_6)=(4{\rm T}3,6{\rm T}9)$, where neither $-3b$ nor $-3b\delta$ is a square.

 Suppose that $f(x)$ is monogenic. Then, from Corollary \ref{Cor:KKR}), $g_2(x)$, $g_4(x)$ and $g_6(x)$ are all monogenic, with $b$ and $\W$ squarefree. From Theorem \ref{Thm:Main g_4}, we have that $b\ne 1$, $b\delta$ is not a square and $(a \mmod{4},\ b \mmod{4})\in \RR$. 
  By Theorem \ref{Thm:KKR}, since $g_6(x)=g_2(x^3)$, and $g_2(x)$ is monogenic from Theorem \ref{Thm:Main g_2}, we can use Theorem \ref{Thm:JKS} with $q=3$ to determine precise conditions under which $g_6(x)$ is monogenic. Observe that condition \eqref{JKS:C5} of Theorem \ref{Thm:JKS} is satisfied since $\W$ is squarefree. If $3\mid a$ and $3\mid b$, then condition \eqref{JKS:C1} of Theorem \ref{Thm:JKS} is satisfied if and only if $(a \mmod{9},\ b \mmod{9})$ is an element of $\SSS_1$, since $b$ is squarefree. If $3\mid a$ and $3\nmid b$, then straightforward calculations show that condition \eqref{JKS:C2} of Theorem \ref{Thm:JKS} is satisfied if and only if $(a \mmod{9},\ b \mmod{9})\in \SSS_2$.
 Similarly, if $3\nmid a$ and $3\mid b$, then it is easy to see that condition \eqref{JKS:C3} of Theorem \ref{Thm:JKS} is satisfied if and only if $(a \mmod{9},\ b \mmod{9})\in \SSS_3$. 
 Finally, we examine condition \eqref{JKS:C4} of Theorem \ref{Thm:JKS} to get that 
 \[H_1(x)=x^2+ax+b \quad \mbox{and} \quad H_2(x)=\left(\frac{a-a^3}{3}\right)x^3-a^2bx^2-ab^2x+\frac{b-b^3}{3}\] are coprime in $\F_3[x]$ if and only if $(a \mmod{9},\ b \mmod{9})\in \SSS_4$. 
  
 The converse is straightforward and follows by utilizing the calculations done previously in the other direction of the proof of this case.  
 \end{proof}








\end{document}